\documentclass[11pt]{article}

\usepackage[utf8]{inputenc}
\usepackage[T1]{fontenc}
\usepackage[english]{babel}
\usepackage{amsmath,amssymb,amsthm}
\usepackage{authblk}
\usepackage{geometry}
\usepackage{xcolor}
\usepackage{bbm}
\usepackage{centernot}
\usepackage{enumitem}
\usepackage{hyperref}
\usepackage{tikz}
\usepackage{tkz-graph}
\usetikzlibrary{patterns}
\usetikzlibrary{hobby}

\setlist[itemize]{leftmargin=*}
\setlist[enumerate]{leftmargin=*,label=(\roman*),ref=(\roman*)}

\renewcommand{\le}{\leqslant}
\renewcommand{\leq}{\leqslant}
\renewcommand{\ge}{\geqslant}
\renewcommand{\geq}{\geqslant}

\newcommand{\bbE}{{\ensuremath{\mathbb E}} }

\newcommand{\bbP}{{\ensuremath{\mathbb P}} }

\newcommand{\bbR}{{\ensuremath{\mathbb R}} }

\newcommand{\bbZ}{{\ensuremath{\mathbb Z}} }

\newtheorem{theorem}{Theorem}

\newtheorem{lemma}[theorem]{Lemma}

\theoremstyle{definition}
\newtheorem{definition}[theorem]{Definition}
\newtheorem{remark}[theorem]{Remark}

\newtheorem{claim}[theorem]{Claim}

\numberwithin{theorem}{section}

\newcommand{\1}{\mathbbm{1}}
\newcommand{\nto}{\centernot\leftrightarrow}
\newcommand{\lr}{\leftrightarrow}
\newcommand{\step}{\texttt{STEP}}

\newcommand{\fail}{\texttt{FAIL}}
\newcommand{\select}{\texttt{SELECT}}
\newcommand{\expl}{\texttt{EXPLORATION}}

\DeclareMathOperator{\argmin}{arg\,min}
\DeclareMathOperator{\essinf}{ess\,inf}
\DeclareMathOperator{\Exp}{Exp}

\title{Weighted isoperimetry implies percolation}

\author[1]{Ivailo Hartarsky}
\author[2]{Franco Severo}
\author[3]{Augusto Teixeira}
    \affil[1]{\small CNRS, IRL2924 Jean-Christophe Yoccoz, IMPA, Estrada Dona Castorina, 110 - Rio de Janeiro, Brazil, %Universit\'e Lyon 1, Centrale Lyon, INSA Lyon, Universit\'e Jean Monnet, CNRS, ICJ UMR5208, 69622 Villeurbanne, France, 
    \texttt{hartarsky@math.univ-lyon1.fr}}
    \affil[2]{\small CNRS, Sorbonne Université, 4 place Jussieu, 75005 Paris, France, 
    \texttt{severo@lpsm.paris}}
\affil[3]{\small IMPA, Estrada Dona Castorina, 110 - Rio de Janeiro, Brazil, \texttt{augusto@impa.br}}

\date{\today}

\begin{document}
\maketitle

\begin{abstract}
Consider an infinite edge-weighted graph satisfying an isoperimetric inequality of the type $\|\partial A\|\ge C|A|^\alpha$ for some $\alpha,C>0$, where $\|\partial A\|$ denotes the weighted size of the edge boundary of $A$.
We prove that, for $C$ large enough depending on $\alpha$, if each edge is open independently with probability given by its weight, then any vertex is connected to infinity with positive probability. The result also holds under weaker isoperimetric assumptions and on finite graphs.

The proof brings a new perspective on the recent proof of the Benjamini--Schramm conjecture concerning the same problem with homogeneous weights. 
The crucial novelty in our proof is that, rather than simply counting cutsets, we introduce a new Peierls argument which takes into account internal and external connectivity costs in addition to the cost of the blocking surface.

We provide two applications for the above result.
First, we show that every non-summable long-range percolation on $\bbZ^d$, $d\geq 2$, admits a percolating truncation, solving a conjecture of Sidoravicius, Surgailis and Vares and its generalization by Friedli and de Lima. Secondly, we show that there exists a universal constant $C < \infty$ such that $p_{\mathrm{c}} \leq C/\Delta$ for every transitive graph of superlinear growth and vertex degree $\Delta$, thus proving a conjecture of Easo and Hutchcroft. 
\end{abstract}


\section{Introduction}
\label{sec:intro}

\subsection{Main result}
Let $V$ be either a finite or an infinite countable set and $\lambda=(\lambda_{\{x,y\}})_{\{x,y\}\subset V}$ be a family of non-negative weights. One can see $G=(V,\lambda)$ as a weighted graph, where the edge set is given by $E:=\{\{x,y\}:~\lambda_{\{x,y\}}>0\}$. 
For any set of edges $F\subset E$, we define its total weight by $\|F\|_\lambda:= \sum_{f\in F} \lambda_f$. 
Consider the \emph{weighted isoperimetric profile}
$\psi_\lambda : \{1,\ldots, \lfloor \tfrac{|V|}{2} \rfloor \} \to [0,\infty]$ given by
\begin{equation}\label{eq:psi}
    \psi_\lambda(n):= \inf \Big\{ \|\partial S\|_\lambda :~S\subset V,\, n\leq |S|\leq \min(2n,|V|/2) \Big\},
\end{equation}
where $\partial S:=\{\{x,y\}\in E: x\in S, y\notin S\}$ is the \emph{edge boundary} of $S$. Given two vertices $x,y\in V$, we set
\begin{align}
  \label{eq:chi}
  \chi_\lambda(x,y)&{}=\inf\{\|\partial S\|_\lambda\colon x\in S\subset V\setminus \{y\}\},&\chi_\lambda(x)&=\inf\{\|\partial S\|_\lambda\colon x\in S\text{ finite}\}.
\end{align}
    
We also consider the associated \emph{bond percolation} model on $G=(V,\lambda)$ where each edge $e \in E$ is declared open independently with probability $p_e := 1 - e^{-\lambda_e}$. We write $\omega\subset E$ for percolation configurations and denote the corresponding probability measure by $\bbP_\lambda$. Given two vertices $x,y\in V$, we denote by $x\lr y$ the event that there exists a connected component $C$ of $(V,\omega)$ such that $\{x,y\}\subset C$. We write $x\lr\infty$ if $x\lr y$ occurs for infinitely many $y$.

Our main result is the following.

\begin{theorem}[Weighted isoperimetry implies percolation]
\label{th:main}
Let $G=(V,\lambda)$ be a weighted graph such that
\begin{equation}\label{eq:assumption}
  \Psi^{-1}:=\sum_{k\geq0} \frac{1}{\psi_\lambda(2^k)} < \frac{1}{625}
\end{equation}
Then the following holds, where  $\varepsilon=\varepsilon(\Psi):=25/\sqrt{\Psi}<1$.
\begin{enumerate}
    \item\label{item:finite} If $V$ is finite, then for every $x,y\in V$, we have
\[\bbP_\lambda(x\lr y)\geq 1- e^{-(1-\varepsilon)\chi_\lambda(x,y)
}.\]
    \item\label{item:infinite} If $V$ is infinite, then for every $x\in V$, we have
\[\bbP_\lambda(x\lr \infty)\geq 1- e^{-(1-\varepsilon)\chi_\lambda(x)}.\]
\end{enumerate}
\end{theorem}

\begin{remark}[Isoperimetric profile]
\label{rem:psi~phi}
The function $\psi_\lambda$ in \eqref{eq:psi} is related to the following more standard function
\begin{equation}\label{eq:phi}
    \phi_\lambda(n):= \inf \Big\{ \|\partial S\|_\lambda :~S\subset V,\, n\leq |S|\leq |V|/2 \Big\}.
\end{equation}
Obviously, $\phi_\lambda(n)=\inf_{m\geq n}\psi_\lambda(m) \leq \psi_\lambda(n)$, so Theorem~\ref{th:main} also holds with $\psi_\lambda$ replaced by $\phi_\lambda$. However, for technical reasons, we will use the version with $\psi_\lambda$ when deducing Theorem~\ref{th:transitive} below. 
\end{remark}

\begin{remark}[Isoperimetric dimension]\label{rem:d>1}
Assumption \eqref{eq:assumption} is satisfied, in particular, if there exists $\alpha\in(0,1]$ such that $\psi_{\lambda}(n)\geq \frac{1250}{\alpha} n^\alpha$ for all $n\geq1$. Recall that when $\psi_{\lambda}(n)\geq c n^\alpha$ for some $c>0$, we say that $G$ satisfies a (weighted) \emph{isoperimetric inequality of dimension $d=1/(1-\alpha)$}.
\end{remark}

\begin{remark}[Surface-order exponential decay]\label{rem:surf_decay}
As it is common for perturbative results, our proof also yields surface-order exponential decay. For instance, if $V$ is infinite and \eqref{eq:assumption} holds, then for all $x\in V$
\[\bbP_\lambda(n\le |\{z\in V:x\lr z\}|<\infty)\le e^{-(1-\varepsilon)\phi_\lambda(n)}.\]
\end{remark}

\begin{remark}[Giant and connectivity for finite graphs]
For finite graphs, one may use Theorem~\ref{th:main}\ref{item:finite} (see also Remark~\ref{rem:surf_decay}) to deduce an upper bound on the threshold for the emergence of a unique giant component from isoperimetric information. Furthermore, a union bound over $y$ in Theorem~\ref{th:main}\ref{item:finite} can be used to recover sharp full connectivity results. For many regular graphs of degree $d$ and homogeneous weights $\lambda_0\1_E$, one can obtain a giant for $\lambda_0= C/d$, and full connectivity for $\lambda_{0}=(1+\epsilon)\log|V|/d$, where the latter is often sharp as it is related to the non-existence of isolated vertices. For instance, this applies to complete graphs, hypercubes, random regular graphs, etc. We refer to \cite{DEKK,Joo,DG} for recent related works.
\end{remark}

A homogeneous version of Theorem~\ref{th:main} (that is for $\lambda_e=\lambda_0 \1_{e\in E}$ for some $\lambda_0>0$) was recently proved by an AI \cite{Cut26a,Cut26b}, thereby settling, in particular, the ``$p_{\mathrm{c}}<1$ question'' of Benjamini and Schramm \cite[Question~2]{BenSch96} on graphs with isoperimetric dimension $d>1$. See also the predecessors \cite{DGRSY20} and \cite{EST25} where this conjecture was verified for $d > 4$ and $d > 2$, respectively.

\subsection{Applications}\label{sec:applications}

We next gather two applications of Theorem~\ref{th:main}. Let us remark that \cite{Cut26a,Cut26b} do not suffice for the applications below, including in the case of homogeneous weights. Indeed, as we explain in Section~\ref{sec:ideas}, their proof is based on cutset counting, which necessarily degenerates as $\lambda\to 0$.

\paragraph{Truncation of long-range percolation.}
We first turn to solving the truncation problem for non-summable long-range percolation, using Theorem~\ref{th:main}.
\begin{theorem}[Percolating truncation always exists]
\label{th:truncation}
  Fix the vertex set $V=\bbZ^d$, with $d \ge 2$. Let $\lambda=(\lambda_{\{x,y\}})_{x,y\in \bbZ^d}$ be   translation-invariant family of weights, that is, $\lambda_{\{x,y\}}=\lambda_{y-x} =\lambda_{x-y}$ for all $x,y\in\bbZ^d$. Assume that $X:=\{x\in\bbZ^d: \lambda_{x}>0\}$ is not contained in the line $\bbR x$ for any $x\in\bbR^d$ and that $\sum_{x\in X}\lambda_{x}=\infty$. Then there exists $M\in[1,\infty)$ such that, setting $\lambda^M_{\{x,y\}}=\lambda_{\{x,y\}}\1_{\|x-y\|_1\le M}$, we have
\[\bbP_{\lambda^M}(0\lr\infty)>0.\]
\end{theorem}

Theorem~\ref{th:truncation} was first conjectured by Sidoravicius, Surgailis and Vares \cite{SSV} in 1999 on $\bbZ^2$ with additional hypotheses, but remained open until now. Over the years, several particular cases of Theorem~\ref{th:truncation} were proved \cite{SSV,MSV,FdLS,FdL,dLS,CdL,Bau} assuming various degrees of symmetry or structure in the weights $\lambda$ and stronger non-summability or $d\ge 3$ assumptions. The most general previous result (though not encompassing the others) is the one of \cite{Bau} proving Theorem~\ref{th:truncation} when $d\ge 3$ and $\lambda$ is invariant under reflection in each of the coordinate hyperplanes. We also direct the reader to \cite{vEdLV,AHdLV} for oriented versions of the problem.

\paragraph{Critical probability of transitive graphs. }
Our second application of Theorem~\ref{th:main} concerns percolation on transitive graphs of large degree (such as high-dimensional lattices, spread-out percolation). Recall that the critical point (in standard parametrization) for bond percolation on an infinite connected graph $G=(V,E)$ can be defined as \[p_{\mathrm{c}}^{\mathrm{bond}}(G):=\inf \Big\{p\in[0,1]:\bbP_{-\log(1-p)\1_E}(x\lr\infty)>0\Big\},\]
where the definition does not depend on $x\in V$.
\begin{theorem}[Large average degree implies percolation]
\label{th:transitive}
There exists a universal constant $C<\infty$ such that if $G=(V,E)$ is a transitive graph with superlinear growth\footnote{A transitive graph $G$ is said to have \emph{superlinear growth} if $\sup_{n\ge 1}|B_n^G|/n=\infty$, where $B_n^G$ is a ball of radius $n$.} and vertex degree $\Delta$, then $p_{\mathrm{c}}^{\mathrm{bond}}(G)\leq 1-e^{-C/\Delta}\leq C/\Delta$.
\end{theorem}

Theorem~\ref{th:transitive} confirms \cite[Conjecture 7.3]{EH23}.  
The fact that $p_{\mathrm{c}}(G)<1$ for transitive graphs of superlinear growth was conjectured in \cite{BenSch96} and first proved in \cite{DGRSY20}. The existence of a gap at $1$ (i.e.~a universal $\varepsilon>0$ such that $p_{\mathrm{c}}(G)\leq 1-\varepsilon$ for every such graph), was first proved for Cayley graph in \cite{PS23} (including site percolation), and for transitive graphs in \cite{EST25} (only for bond percolation). In fact, combining the well-known bound $p_{\mathrm{c}}^{\mathrm{site}}(G)\leq 1-(1-p_{\mathrm{c}}(G))^{\Delta-1}$ from \cite{GriSta98} with our Theorem~\ref{th:transitive} establishes a gap at $1$ for site percolation on transitive graphs, which was open until now.

We recall the classical fact that $p_{\mathrm{c}}(G)\geq 1/(\Delta-1)$ for any graph of maximum degree $\Delta$. Asymptotic analogues of Theorem~\ref{th:transitive} were proved with a constant $C\to 1$ as $\Delta\to\infty$ for spread-out percolation on $\bbZ^d$ \cite{Pen,BJR} and on transitive graphs of polynomial growth \cite{ST}. However, it is known \cite[Figure~5]{EH23} that $C$ cannot be taken arbitrarily close to $1$ for general transitive graphs, even as $\Delta\to\infty$.

\subsection{Outline of the proof}\label{sec:ideas}

We now discuss the main ideas in the proof of Theorem~\ref{th:main}. For simplicity we assume that $V$ is finite (the case infinite $V$ follows by restricting and wiring the boundary), and we aim at upper bounding $\bbP[x\nto y]$ uniformly in $x,y\in V$. 

In order to better appreciate the challenges and innovations of this work, we start by recalling the approach introduced in \cite{EST25} and further developed in \cite{Cut26a,Cut26b}, which is specific to the homogeneous case $\lambda=\lambda_0\1_E$. The first step, which can be traced back to the work of Peierls \cite{Pei36}, is to reduce the problem to counting (minimal) cutsets in the graph $G=(V,E)$. More precisely, first notice that the event $\{x\nto y\}$ exactly corresponds to the existence of a cutset $\pi\subset E$ separating $x$ from $y$ (see Definition~\ref{def:cutset}) which is $\omega$-closed (i.e.~$\omega\cap\pi=\emptyset$). In particular, $\bbP_\lambda[x\nto y]\leq\sum_{\pi}\bbP_\lambda[\pi \text{ is } \omega\text{-closed}]$. The second step is to prove that there exists a constant $C<\infty$ such that $|\{\pi:~|\pi|\leq n\}|\leq e^{Cn}$ for all $n$.  Since $\bbP[\pi \text{ is } \omega\text{-closed}]=e^{-\|\pi\|_\lambda}=e^{-\lambda |\pi|}$ in the homogeneous case, the result would follow readily by taking $\lambda_0>C$ large enough. 
However, counting cutsets on general graphs can be a very challenging combinatorial problem. The main idea of \cite{EST25} is to do so via a probabilistic method, namely by carefully constructing a random cutset $\Pi$ such that $\bbP[\Pi=\pi]\geq e^{-C|\pi|}$ for all $\pi$. Such a random cutset is built in \cite{EST25} by using the simple random walk on $G$ and assuming that it is transient (which is known to be implied by an isoperimetric inequality of dimension $d>2$). In \cite{Cut26a,Cut26b}, such a random cutset was obtained via a clever modification of Karger's algorithm \cite{Kar}, and only required a summability assumption such as \eqref{eq:assumption}. The algorithm is very simple: start with the vertex set $V$, where each site is seen as a singleton-cluster, then at each step select a cluster \emph{of smallest size} out of those not containing $x$ or $y$, take a uniformly chosen random edge between this cluster and its neighboring ones, and then join them into a single cluster. Iterating this process, one ends up with only two clusters (of $x$ and $y$), and $\Pi$ is precisely the edges between them.

For general inhomogeneous weights such as in Theorem~\ref{th:main}, estimating the expected number of $\omega$-closed cutsets as above has no chance to work, even in the particular contexts of Theorems~\ref{th:truncation} and \ref{th:transitive}. Indeed, consider the cutsets $\pi$ (between $x$ and $\infty$) obtained from the edges incident to $x$ and $z$, for all vertices $z$ such that $\{x,z\}\in E$. If $\lambda=\lambda^M$ as in Theorem~\ref{th:truncation}, these cutsets satisfy $\|\pi\|_{\lambda^M}\leq 2\Lambda_M$, where $\Lambda_M:=\sum_{\|w\|_1 \leq M} \lambda_w$. If $\lambda_w>0$ for all $w$, there are at least $|B_M^{\bbZ^d}|-1$ such cutsets, thus $\sum_{\pi}\bbP_{\lambda^M}[\pi \text{ is } \omega\text{-closed}]\geq cM^d e^{-2\Lambda_M}$, which may diverge if $\Lambda_M=o(\log M)$ as $M\to\infty$. In the context of Theorem~\ref{th:transitive}, setting $\lambda:=\frac{C}{\Delta}\1_E$, we have $\|\pi\|_{\lambda}\leq 2C$ for these cutsets, thus $\sum_{\pi}\bbP_{\lambda}[\pi \text{ is } \omega\text{-closed}]\geq \Delta e^{-2C}$, which diverges as $\Delta\to\infty$. One can check that such divergences may occur not only for these small cutsets, but also for arbitrarily large ones.

The main idea of our proof is to introduce the notion of \emph{cohesion}.
We say that a cutset $\pi$ is $\omega$-cohesive if, in addition to being $\omega$-closed, none of its sides contains a small cut (i.e.~with total weight $<1$) which is not crossed by $\omega$ (see Definition~\ref{def:cohesive}). 
First, we observe that $\{x\nto y\}$ happens if and only if there exists a $\omega$-cohesive cutset (see Lemma~\ref{lem:existence}). 
We then use the same algorithm as the one of \cite{Cut26a} described above, but where edges are selected with probability proportional to their weights.
This yields a random cutset $\Pi$ for which we prove that 
\begin{equation}\label{eq:goal}
\bbP[\Pi=\pi]\ge e^{(1-\varepsilon) \|\pi\|_\lambda} \bbP_\lambda[\pi \text{ is } \omega\text{-cohesive}]
\end{equation} 
for all $\pi$. Since by \eqref{eq:chi} we have $\|\pi\|_\lambda \geq \chi_\lambda(x,y)$ for all $\pi$, the bound \eqref{eq:goal} readily implies $\bbP_\lambda[x\nto y]\le \sum_{\pi}\bbP_\lambda[\pi \text{ is } \omega\text{-cohesive}] \le \sum_{\pi} e^{-(1-\varepsilon)
\|\pi\|_\lambda}\bbP[\Pi=\pi]\le e^{-(1-\varepsilon)\chi_\lambda(x,y)}$, thus concluding the proof.

In order to prove \eqref{eq:goal}, we couple the algorithm with the percolation process via the following Poissonazation trick. Consider the edge-marks $(P_e)_{e\in E}$ given by independent exponential random variables with parameters $(\lambda_e)_{e\in E}$. First, note that $\omega=\1_{P_e\leq1}$ has law $\bbP_\lambda$. Second, the law of the aforementioned algorithm can be obtained from the same $(P_e)_{e\in E}$ by merging along the edge on the boundary of the selected cluster with minimal mark. In order to preserve the exact law of the algorithm, at each step, we erase some negative information acquired on remaining edges. When estimating the probability that $\Pi=\pi$, we note that there are two kinds of steps: either the non-$\pi$ boundary of the selected cluster has total weight at least $1$, or smaller than~$1$. On the one hand, the probability of not merging along $\pi$ for all steps of the first type can be estimated exactly as in \cite{Cut26b}, using the assumption \eqref{eq:assumption} on the isoperimetric profile, thus leading to a lower bound of the form $e^{-\varepsilon\|\pi\|_\lambda}=e^{(1-\varepsilon)}\|\pi\|_\lambda\bbP[\pi \text{ is } \omega\text{-closed}]$. On the other hand, for steps of the second type, our construction essentially guarantees that the selected edge does not belong to $\pi$ when it is $\omega$-cohesive. As a consequence, the second type of steps generates an additional probability cost, whose total can be lower bounded by $\bbP_\lambda[\pi \text{ is } \omega\text{-cohesive} | \pi \text{ is } \omega\text{-closed}]$. These two estimates combined imply \eqref{eq:goal}. In retrospect, this construction of the algorithm is closer to Kruskal's minimum spanning tree algorithm \cite{Kru}, which actually inspired \cite{Kar}.

Finally, let us mention that deducing Theorems~\ref{th:truncation} and \ref{th:transitive} from Theorem~\ref{th:main} amounts to proving a simple isoperimetric inequality in two dimensions (Lemma~\ref{prop:truncation}) and deducing one  (Lemma~\ref{lem:transitive}) on general transitive graphs from geometric group theoretic results \cite{TT24,TT20}. These proofs are short and presented in Sections~\ref{sec:truncation} and \ref{sec:transitive}, respectively.

\section{Proof of Theorem~\ref{th:main}}
\label{sec:main}
In this section we prove Theorem~\ref{th:main} starting with the finite graph case. Assume $V$ is finite and fix $x,y\in V$.
For the purposes of dealing with infinite graphs later, we will work with the following (weaker) version of $\psi$ avoiding $x$ and $y$:
\begin{equation}
  \label{eq:def:psi:star}
  \psi^*_\lambda(n)
  = \inf \Big\{ \|\partial S\|_{\lambda}: S \subset V \setminus \{x, y\}, n \le |S| \le \min(2n, |V|/2) \Big \}
  \ge \psi_\lambda(n),
\end{equation}
for $n \in\{ 1, \dots, \lfloor |V|/2 \rfloor\}$. 
The proof will only rely on the fact that 
\begin{equation}\label{eq:def_Psi^star}
    (\Psi^*)^{-1}:=\sum_k 1/\psi^*_\lambda(2^k) < 1/25^2
\end{equation}
and
$\phi_\lambda(1)=\inf_k \psi_\lambda(2^k)\geq \Psi > 25^2$, 
which are both consequences of our assumption \eqref{eq:assumption}.

\begin{definition}[Cutset]
  \label{def:cutset}
  A \emph{cutset} (between $x$ and $y$) is an inclusion minimal subset $\pi$ of $E$ with the property that $x$ and $y$ are not connected in the graph $(V,E\setminus\pi)$. Given a cutset $\pi$, we denote by $V_x$ and $V_y$ the vertex sets of the connected components in $(V,E\setminus\pi)$ containing  $x$ and $y$, respectively. Let $E_x$ and $E_y$ denote the edge sets of the graphs induced by $V_x$ and $V_y$, respectively. Note that, by minimality, for any cutset $\pi$, we have $V_x\sqcup V_y=V$.
\end{definition}

\begin{definition}[$\omega$-cohesive cutset]
\label{def:cohesive}
Given $\omega \subset E$, a cutset $\pi$ is said to be $\omega$-\emph{cohesive}, if the following events occur.
\begin{itemize}
  \item $\pi\cap\omega=\varnothing$,
  \item for any $z\in\{x,y\}$ and $F\subset E_z$ such that $(V_z,E_z\setminus F)$ is disconnected and $\|F\|_\lambda< 1$, we have $\omega\cup F\neq \emptyset$.
\end{itemize}
\end{definition}
In words, an $\omega$-cohesive cutset is one which is closed and such that, in both sides of it, all small cuts meet $\omega$ (see Figure~\ref{fig:cutset}).

\begin{remark}[Cohesion]
To digest Definition~\ref{def:cohesive}, let us consider the homogeneous case of $\lambda_e=\lambda_0$ for all $e\in E$. If $\lambda_0>1$, a cutset is $\omega$-cohesive, iff it is closed (the second condition is void). If $\lambda_0\in(1/2,1)$, a cutset is $\omega$-cohesive, iff it is closed and all bridges (in the graph theoretic sense of the term) of $(V,E\setminus\pi)$ belong to $\omega$. If $G$ is planar and $\lambda_0\in(1/2,1)$, $\omega$-cohesive cutsets correspond to \emph{induced cycles} (also known as holes) separating $x$ and $y$ in the dual configuration $(V^*,\omega^*)$, defined by $\omega^*_{e^*}=1-\omega_e$. Lower values of $\lambda_0$ correspond to higher order notions involving pockets accessible by multiple edges.
\end{remark}

\begin{figure}
  \centering
  \includegraphics[width=0.6\textwidth]{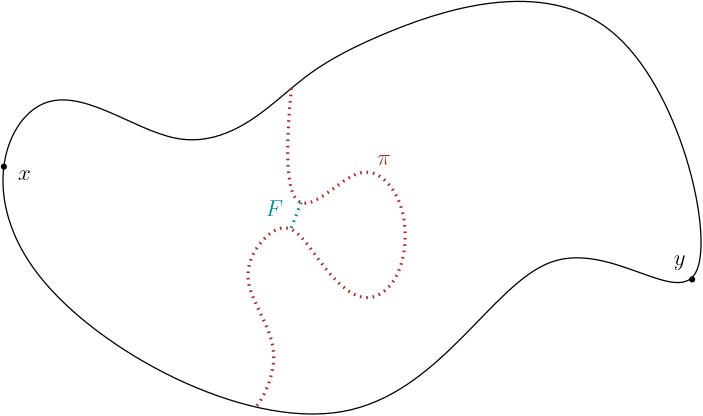}
  \caption{An illustration of a small-weight set $F$ that disconnects $V_x$.
  Assuming that that $\| F \|_\lambda < 1$, then $F$ must intersect $\omega$ in order for $\pi$ to be considered $\omega$-cohesive.}
  \label{fig:cutset}
\end{figure}

\begin{lemma}[$\omega$-cohesive cutsets exist]
\label{lem:existence}
Assume that $\inf_k \psi^*_\lambda(2^k) \ge 2$. Let $\omega\subset E$ be such that $x\nto y$ in $\omega$. Then there exists an $\omega$-cohesive cutset.
\end{lemma}

\begin{proof}
The proof roughly says that the minimum-weight cutset not meeting $\omega$ is $\omega$-cohesive, because any small cut in $V_x$ allows us to switch one of its sides to $V_y$, yielding a cutset with smaller weight.

Let $\pi$ be any minimum-weight closed cutset. More precisely, let $\pi$ be some cutset for which,
\begin{equation}
  \label{eq:min:cut}
  \|\pi\|_{\lambda}=\min\{\|\pi'\|_{\lambda}:\pi'\subset E\setminus\omega,\pi'\text{ cutset}\}.
\end{equation}
The set on the right is nonempty, as we can start with $E\setminus\omega$ and delete edges one at a time, keeping $x$ and $y$ disconnected. Hence, $\pi$ exists and $\pi\cap\omega=\varnothing$ by definition. We claim that $\pi$ is $\omega$-cohesive.

Assume by contradiction that there exists $F\subset E_x\setminus\omega$ with $G_x=(V_x,E_x\setminus F)$ disconnected and $\|F\|_{\lambda}<1$. Let $V'_x$ be the connected component of $x$ in $G_x$ and $W_x=V_x\setminus V'_x\neq\varnothing$. Let $\partial W_x=\{\{a,b\}\in E:a\in W_x,b\in V\setminus W_x\}$ be the boundary of $W_x$ in $G$. By construction $\partial W_x\subset \pi\cup F$. Consider the set $\pi'=\pi\Delta \partial W_x\subset\pi\cup F$, where $\Delta$ is the symmetric difference. Then $\pi'\cap\omega=\varnothing$ and $x\in W_x$ is not connected to $y \in V_{y}$ in $(V,E\setminus\pi')$. Hence, $\pi'$ contains a cutset and \eqref{eq:min:cut} gives $\|\pi'\|_{\lambda}\ge\|\pi\|_{\lambda}$. However,
\[\|\pi'\|_{\lambda}=\|\pi\|_{\lambda}+2\|\partial W_x\setminus\pi\|_{\lambda}-\|\partial W_x\|_{\lambda}\le\|\pi\|_{\lambda}+2\|F\|_{\lambda}-2<\|\pi\|_{\lambda},\]
since $\|\partial W_x\|\ge \inf_k \psi^*_\lambda(2^k)
\geq 2$, as $W_x$ is nonempty and does not contain $x$ or $y$.
This contradiction (and its analogue for $G_y$) concludes the proof.
\end{proof}

Similarly to \cite{Cut26a}, we will construct an algorithm which starts from $G$ and sequentially collapses edges according to some random input. This process will generate a sequence of multi-graphs, with the last one consisting of the two points $\{x, y\}$ joined by some edges. The edges present in the final multi-graph will form a (random) cutset between $x$ and $y$. In order to describe a given step of this procedure, we need to setup some notation for multi-graphs with (edge and vertex) weights.

Fix a connected multi-graph $H = (V_H, E_H)$ with $x,y\in V_H$, endowed with the following functions.
The first one is an edge-weight $\lambda^H : e \in E_H \mapsto \lambda^H_e \in [0,\infty)$ which will represent the percolation weights of each edge.
We recall that $H$ can be a multi-graph and, in this case, the function $\lambda^H$ can assign different weights to distinct edges connecting the same endpoints.
We systematically erase self-loops from multi-graphs. 
The second function we are given is a vertex-weight $|\cdot|_H : V_H \to \bbR$, which encodes the size of the cluster represented by this vertex.
Fix also $P^H : E_H \to (0, \infty) : e \mapsto P^H_e$ which, roughly speaking, associates to each edge a ``clock ring'', marking the moment at which this edge will collapse.

The next definition was one of the main inputs from \cite{Cut26a}.

\begin{definition}[\select]
\label{def:select}
Let $H=(V_H,E_H)$ be a multi-graph with a vertex-weight function $|\cdot|_H:V_H\to\bbR$. The procedure \select{} returns some $v_H\in\argmin_{V_H\setminus \{x,y\}}|\cdot|_H$, breaking ties in an arbitrary deterministic way.
\end{definition}

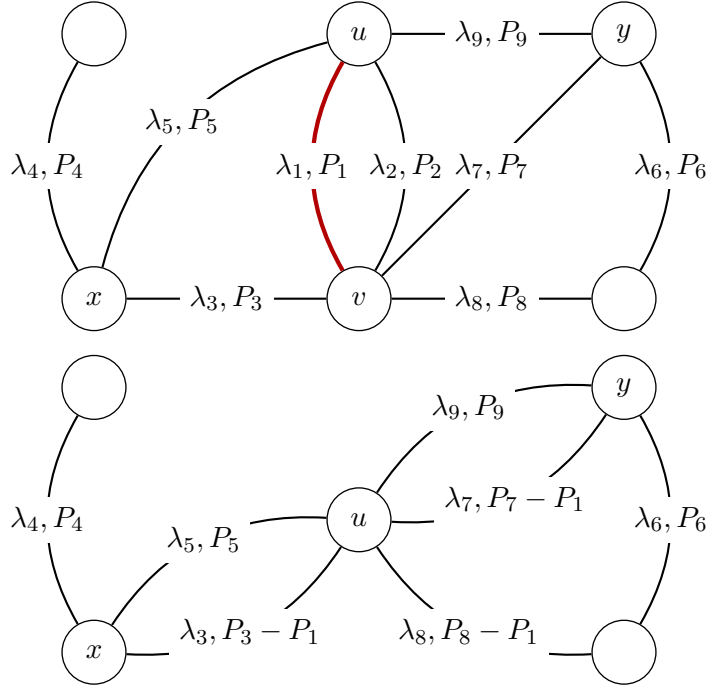
\begin{figure}
    \centering
    \begin{tikzpicture}[x=3.5cm,y=3.5cm]
    \draw[color=white] (0, 0) -- (0, -.2);
    \GraphInit[vstyle=Normal]
    \tikzset{VertexStyle/.append style = {minimum size = 24pt}}
\Vertex[x=0,y=0,L=$x$]{a}
\Vertex[x=1,y=0,L=$v$]{b}
\Vertex[x=1,y=1,L=$u$]{c}
\Vertex[x=0,y=1,NoLabel]{d}
\Vertex[x=2,y=0,NoLabel]{e}
\Vertex[x=2,y=1,L=$y$]{f}

\Edge[label={$\lambda_2,P_2$},style={bend left}](c)(b)
\Edge[label={$\lambda_3,P_3$}](a)(b)
\Edge[label={$\lambda_4,P_4$},style={bend left}](a)(d)
\Edge[label={$\lambda_5,P_5$},style={bend left}](a)(c)
\Edge[label={$\lambda_6,P_6$},style={bend right}](e)(f)
\Edge[label={$\lambda_7,P_7$}](b)(f)
\Edge[label={$\lambda_8,P_8$}](b)(e)
\Edge[label={$\lambda_9,P_9$}](c)(f)

\tikzset{EdgeStyle/.append style = {ultra  thick,color=red!70!black}}
\Edge[label={$\lambda_1,P_1$},style={bend right}](c)(b)
    \end{tikzpicture}

    \begin{tikzpicture}[x=3.5cm,y=3.5cm]
    \GraphInit[vstyle=Normal]
    
  \tikzset{VertexStyle/.append style = {minimum size = 24pt}}
\Vertex[x=0,y=0,L=$x$]{a}
\Vertex[x=1,y=0.5,L=$u$]{b}
\Vertex[x=0,y=1,NoLabel]{d}
\Vertex[x=2,y=0,NoLabel]{e}
\Vertex[x=2,y=1,L=$y$]{f}

\Edge[label={$\lambda_3,P_3-P_1$},style={bend right}](a)(b)
\Edge[label={$\lambda_4,P_4$},style={bend left}](a)(d)
\Edge[label={$\lambda_5,P_5$},style={bend left}](a)(b)
\Edge[label={$\lambda_6,P_6$},style={bend right}](e)(f)
\Edge[label={$\lambda_7,P_7-P_1$},style={bend right}](b)(f)
\Edge[label={$\lambda_9,P_9$},style={bend left}](b)(f)
\Edge[label={$\lambda_8,P_8-P_1$},style={bend right}](b)(e)
    \end{tikzpicture}
    \caption{Illustration of \step{} of Definition~\ref{def:step} applied to the multi-graph $H$ above (with indices $H$ suppressed) resulting in the multi-graph $H'$ below. The thick red edge joining $u$ and $v$ is contracted. The size of $u$ in $H'$ is the sum of those of $u$ and $v$ in $H$.}
    \label{fig:step}
\end{figure}

The next definition incorporates some new features with respect to \cite{Cut26a}.
It is illustrated in Figure~\ref{fig:step}.

\begin{definition}[\step]
\label{def:step}
The procedure \step{} takes as input $H, \lambda^H, |\cdot|_H, P^H$ and $v_H$ as above and returns a multi-graph $H'$ with its own $\lambda^{H'},|\cdot|_{H'},P^{H'}$ but possibly with $E_{H'} = \varnothing$, if $V_{H'}$ is a singleton.
\step{} is also allowed to return a \fail{} message, while still outputting $H',\lambda^{H'},|\cdot|_{H'},P^{H'}$. The procedure is defined as follows (see Figure~\ref{fig:step}).
\begin{itemize}
  \item Let $e_H=\{v_H,u_H\}\in E_H$ be such that $P^H_{e_H} = \min \{P^H_e; e\in E_H \text{ with } v_H \in e\}$. Below, $e_H$ will be a.s.\ uniquely-defined, so ties may be broken arbitrarily if needed at this point.
  \item If $\Lambda_H:=\sum_{e\in E_H; v_H\in e} \lambda^H_e <1$ and $P^H_{e_H}>1$, the algorithm will \fail, but still continue.
  \item Let $H'$ be the multi-graph obtained by identifying $v_H$ and $u_H$, keeping the label $u_H$ for the resulting vertex in $V_{H'}$.
  \item For $v\in V_H\setminus \{v_H,u_H\}$, we set $|v|_{H'}=|v|_H$, while we set $|u_H|_{H'}=|v_H|_H+|u_H|_H$.
  \item Let $P^{H'}_{e}=P^H_e$, if $v_H\not\in e\in E_{H}$, and $P^{H'}_{e'}=P^H_{e}-P^H_{e_H}$ for $e\in E_{H}$ with $v_H\in e$ and $u_H\notin e$, where $e'$ is the image of $e$ through the quotient.
\end{itemize}
\end{definition}

Intuitively, the last step in the above procedure corresponds to letting the time elapse for the other edges adjacent to $v_H$.

\begin{lemma}[Law of \step]
\label{lem:law}
  Assume that, in the input of \step, $(P^H_e)_{e\in E_H}$ are independent, exponential random variables with rates $(\lambda^H_e)_{e\in E_H}$.
  Then,
  \begin{gather}
    \label{eq:law:e_H_distr}
    \text{$\bbP(e_H=e)=\lambda_e^H/\Lambda_H$ for $e\in E_H$ with $v_H \in e$,}\\
    \label{eq:law:fail_bound}
    \text{$\bbP(\fail)\ge \max(0,1-\Lambda_H)$,}\\
    \label{eq:law:e_H_indep_fail}
    \text{$\fail$ is independent of $e_H$,}\\
    \label{eq:law:clock_exponential}
    \text{conditionally on $e_H$ and $\1_\fail$, $(P^{H'}_e)_{e\in E_{H'}}$ are also distributed as $\otimes_{e\in E_{H'}} \Exp(\lambda^{H'}_e)$,}\\
    \label{eq:law:clock_decrease}
    \text{$P^{H'}_{e'}\le P^H_e$, whenever $e'$ is the image of $e$ under the quotient mapping $H$ to $H'$.}
  \end{gather}
\end{lemma}

\begin{proof}
Firstly, $e_H$ is clearly a.s.\ uniquely defined, since the exponential distribution is diffuse.
Thus,
\[\bbP(\fail)=e^{-\Lambda_H} \cdot \1_{\Lambda_H<1}\ge \max(0,1-\Lambda_H).\]
The fact that $(P^{H'}_e)_{e\in E_{H'}}$ have the claimed conditional law follows from the memoryless property of exponential variables. The distribution of $e_H$ is a standard property of Poisson processes. Finally, the fact that $P^{H'}_{e}\le P^H_e$ is immediate from Definition~\ref{def:step}.
\end{proof}

\begin{definition}[\expl]
Let $(P_e)_{e\in E}$ be independent exponential random variables with rates $(\lambda_e)_{e\in E}$. \expl{} is the following exploration algorithm. Start with the multi-graph $H_0=G$ with constant weight $|\cdot|_{H}\equiv 1$. Let $v_{G}$ be given by \select{}. Apply \step{} to obtain a multi-graph $H_1=H'$. We repeat applying \select{} and \step{} until we reach a multi-graph with exactly two vertices (which are necessarily labeled $x$ and $y$). This defines a sequence of multi-graphs $(H_s)_{s=0}^{S}$.
The output of \expl{} is the set $\Pi\subset E$ of edges which are mapped to the edges remaining in the final multi-graph $H_S$. 
\end{definition}

Note that, if, for any nonempty $A\subsetneq V$, $\|\partial A\|_{\lambda}\ge 1$, then, a.s., \expl{} does not fail. Moreover, the output $\Pi$ is necessarily a cutset.

\begin{remark}[Alternative construction]
  \label{rem:two_constructions}
  An important consequence of Lemma~\ref{lem:law} is that it provides an alternative way to construct the sequence of multi-graphs $(H_t)_{t\geq0}$ from \expl{}, together with the sequence of fail events without the use of the exponential random variables $P_e$.
  More precisely, we can alternate between using \select{} to pick a vertex $v_{G_t}$ and then choose $e_{H_t}$ according to \eqref{eq:law:e_H_distr}.
  After the whole sequence $H_t$ is fixed, one can decide whether each individual step $t$ failed using independent coin tosses with parameters $e^{-\Lambda_{H_t}} \1_{\Lambda_{H_t}<1}$.
  Note however, that this construction is not sufficient for the proof of our main theorem because it does not provide the crucial monotonicity \eqref{eq:law:clock_decrease}.
\end{remark}

\begin{definition}[$\expl_\pi$]
Let $\pi$ be a cutset. Then $\expl_\pi$ is the same as \expl{} but run on with  $\tilde \lambda$ given by $\tilde \lambda_e=0$ (in which case $\tilde{P}_e=\infty$) for every $e\in\pi$ and $\tilde{\lambda}_e=\lambda_e$ otherwise. This generates a sequence of multi-graphs $(G_t)_{t=0}^T$, whose law is denoted by $\mu_\pi$. Note that in this case the output is always $\pi$. We say that $\expl_\pi$ succeeds if it does not $\fail$ at any step.
\end{definition}

\begin{lemma}[Cohesion implies success]
\label{lem:success}
Let $P_e$ be independent exponential random variables with rates $(\lambda_e)_{e\in E}$. Let $\omega=\{e\in E:P_e\le 1\}\subset E$ (which is distributed as $\bbP_\lambda$). If a cutset $\pi$ is $\omega$-cohesive, then $\expl_\pi$ succeeds a.s.
\end{lemma}
\begin{proof}
Let $P_e^t$ denote the exponential variable of the edge $e\in E$ after $t$ \step{}s of $\expl_\pi$, identifying the edge with its images under contraction until it becomes a self-loop (and is removed). By Lemma~\ref{lem:law}, $t\mapsto P_e^t$ is non-increasing and its initial value is at most 1 for $e\in\omega$. Therefore, we may \fail{} only on steps when $v_H\not\in\bigcup_{e\in\omega}e$. But then a failure could only be explained by $\Lambda_H < 1$, which would yield a set of edges $F$ contradicting Definition~\ref{def:cohesive}.
\end{proof}

Recall that $(G_t)_{t=0}^T$ is the sequence of multi-graphs obtained by $\expl_\pi$, and
let $\partial_{G_t} v_{G_t}$ represent the set of edges $e\in E$ whose image in $G_t$ under the quotient contains the selected vertex $v_{G_t}$ (edges mapping to loops are not included).
We write
\begin{align}
\label{eq:def:rtbt}
  r_t &{}= \|\partial_{G_t}v_{G_t}\setminus\pi\|_\lambda,&
  b_t &{}= \|\partial_{G_t}v_{G_t}\|_\lambda,&
  a_t &{}= \|\partial_{G_t}v_{G_t} \cap \pi\|_\lambda,
\end{align}
so that $b_t = r_t + a_t$.

\begin{lemma}[Law of $\Pi$]
  \label{lem:radon}
  Given a cutset $\pi$, recalling that $\Pi$ denotes the cutset generated by $\expl$, we have
  \begin{equation}
    \bbP[\Pi=\pi]=\bbE_{\mu_\pi}\left[\prod_t \frac{r_t}{b_t}\right].
  \end{equation}
\end{lemma}

\begin{proof}
  Consider a deterministic sequence of edges $e_0, e_1, \ldots,e_{T-1}$ (necessarily not in $\pi$) that could be selected by $\expl_\pi$, resulting in the sequence of multi-graphs $g_0, g_1, \ldots, g_T$ (necessarily with output $\pi$).
  The probability that this specific sequence is selected under $\expl$ is
  \begin{equation}
    \prod_t \frac{\lambda_{e_t}}
    {b_t}
    = \prod_t \frac{r_t}    
    {b_t} 
    \prod_t
    \frac{\lambda_{e_t}}{r_t},
  \end{equation}
  the latter product being the probability of the trajectory under $\mu_\pi$. Summing over all sequences concludes the proof.
\end{proof}
By Lemma~\ref{lem:law} (see also Remark~\ref{rem:two_constructions}) for any cutset $\pi$, 
\begin{align}
  \label{eq:success}
  \nonumber
  \bbP( \expl_\pi\text{ succeeds})
  & \overset{\eqref{eq:law:fail_bound}, \eqref{eq:law:e_H_indep_fail}}\leq
  \bbE_{\mu_\pi} \bigg[ \prod_t\min(r_t,1) \bigg]
  = \bbE_{\mu_\pi} \bigg[ \prod_t \Big( \frac{r_t}{b_t} \Big)
  / \Big( \frac{\max(r_t,1)}{b_t} \Big) \bigg]\\
  & \le \frac{\bbE_{\mu_\pi} \Big[ \prod_{t}\frac{r_t}{b_t} \Big]}
  {\essinf_{\mu_\pi}\prod_t\frac{\max(r_t,1)}{b_t}}
  \overset{\text{Lemma}~\ref{lem:radon}}
  = \frac{\bbP(\Pi=\pi)}
  {\essinf_{\mu_\pi}\prod_t\frac{\max(r_t,1)}{b_t}}.
\end{align}
In order to conclude, we need the following lemma adapted from \cite{Cut26b}, whose proof is spelled out in Appendix~\ref{sec:app} for completeness.

\begin{lemma}[Isoperimetic lower bound]
\label{lem:gpt}
Recall the definition of $\Psi^*$ from \eqref{eq:def_Psi^star} and assume that $\phi_\lambda(1)\geq1$.
Then, for any cutset $\pi$, a.s.\ under $\mu_\pi$, we have
\begin{align*}
\prod_t\frac{\max(r_t,1)}{b_t}&{}\ge \exp\left(-\delta
\|\pi\|_{\lambda}\right),&\delta&{}:=\left(\frac{24}{\Psi^*}+\frac{22}{\sqrt{\Psi^*}}+\frac{4\log(e\phi_\lambda(1))}{\phi_\lambda(1)}\right)\end{align*}
\end{lemma}

\begin{proof}[Proof of Theorem~\ref{th:main}]
First note that by the assumption \eqref{eq:assumption} we have $\min(\Psi^*,\phi_\lambda(1))\geq\Psi\geq 25^2$, which readily implies $\delta\leq\varepsilon=25/\sqrt{\Psi}<1$.
Combining \eqref{eq:success} and Lemma~\ref{lem:gpt}, we obtain
\begin{equation}
  e^{\varepsilon\|\pi\|_{\lambda}}\bbP(\Pi=\pi)
  \ge \frac{\bbP(\pi\cap\omega=\varnothing)\bbP(\expl_\pi \text{ succeeds})}
  {\bbP(\pi\cap\omega=\varnothing)}.
\end{equation}
This expression can be estimated using the independence between the events $\{\pi\cap\omega=\varnothing\}$ and $\{\expl_\pi\text{ succeeds}\}$ with Lemma~\ref{lem:success}, giving
\begin{equation}\label{eq:match_and_cohesive}
  e^{\varepsilon\|\pi\|_{\lambda}}\bbP(\Pi=\pi)
  \ge e^{\|\pi\|_{\lambda}}\bbP_{\lambda}(\pi\text{ is $\omega$-cohesive}).
\end{equation}
Using Lemma~\ref{lem:existence} and \eqref{eq:match_and_cohesive}, we obtain
\begin{align*}
  \bbP_{\lambda}(x\nto y) &= \bbP_\lambda(\text{there exist an } \omega\text{-cohesive cutset } \pi) 
  \le \sum_{\pi} \bbP_{\lambda}(\pi \text{ is } \omega\text{-cohesive}) \\ &\le \sum_{\pi} e^{-(1-\varepsilon)
  \|\pi\|_{\lambda}}\bbP(\Pi=\pi) \overset{\eqref{eq:chi}}{\le} e^{-(1-\varepsilon)\chi_\lambda(x,y)}.
\end{align*}
This completes the proof of Theorem~\ref{th:main} for finite graphs.

Finally, we observe that the case where $V$ is infinite follows readily from the case where $V$ is finite. Indeed, for any finite subset $V'\subset V$, we can collapse all the vertices in $V\setminus V'$ into a single ``ghost vertex'' $g$. This naturally defines an associated finite weighted (multi-)graph $(V'\cup\{g\},\lambda')$. By construction, one has  $\bbP_{\lambda'}(x\lr g) = \bbP_{\lambda}(x\lr V\setminus V')$ for any $x\in V'$. Furthermore, one can check that $\psi^*_{\lambda'}(n)\ge\psi_\lambda(n)$ for all $n\in\{1, \ldots , \lfloor \frac{|V'|+1}{2} \rfloor \}$ (note that this is where the relevance of \eqref{eq:def:psi:star} comes from) and $\phi_{\lambda'}(1)\ge\phi_\lambda(1)\ge (\sum_k1/\psi_\lambda(2^k))^{-1}$. Therefore, we can apply the finite graph version of Theorem~\ref{th:main} for $(V'\cup\{g\},\lambda')$ to deduce that 
\[\bbP_{\lambda'}(x\lr g)\geq 1-e^{-(1-\varepsilon)\chi_{\lambda'}(x,g)}\le 1-e^{-(1-\varepsilon)\chi_\lambda(x)}.\] This shows that $\bbP_{\lambda}(x\lr V\setminus V')\geq 1-e^{-(1-\varepsilon)\chi_{\lambda}(x)}$ for every finite $V'\subset V$ such that $x\in V'$. By taking an exhaustion of $V$, we obtain $\bbP_{\lambda}(x\lr \infty)\geq 1-e^{-(1-\varepsilon)\chi_\lambda(x)}$ as desired.
\end{proof}

\section{Proof of Theorem~\ref{th:truncation}}
\label{sec:truncation}

Theorem~\ref{th:truncation} will follow readily from Theorem~\ref{th:main} and the following lemma.

\begin{lemma}[Isoperimetric inequality]\label{prop:truncation}
If $\lambda$ satisfies the assumptions of Theorem~\ref{th:truncation}, then the exist non-negative constants $(C_M)_{M\geq1}$ such that $C_M\to\infty$ as $M\to\infty$ and $\psi_{\lambda^M}(n)\geq C_M \sqrt{n}$ for all $M,n\geq1$.
\end{lemma}

We start by assuming the validity of the above lemma.

\begin{proof}[Proof of Theorem~\ref{th:truncation}]
By Lemma~\ref{prop:truncation}, we can find $M$ such that $\psi_{\lambda^M}(n)\geq 2500 n^{1/2}$. For such $M$ we have $\sum_{k\geq0} \frac{1}{\psi_{\lambda^M}(2^k)}\leq \frac{1}{2500}\sum_{k\geq0} 2^{-k/2} = \frac{1}{2500}(\frac{\sqrt{2}}{\sqrt{2}-1}) < \frac1{625}$ and the result follows from Theorem~\ref{th:main}.
\end{proof}

Finally, we provide the missing ingredient that was postponed.

\begin{proof}[Proof of Lemma~\ref{prop:truncation}]
We distinguish two cases.

\noindent\textbf{Case 1.} Assume that there exists $x\in\bbZ^d$ such that $\sum_{z\in\bbZ^d\cap(\bbR x)}\lambda_{z}=\infty$. Intuitively speaking, this corresponds to the case that, even though $\lambda$ is not supported on a line, there exists a line over which $\lambda$ is not summable. Then we may restrict attention to the case $d=2$ by setting $\lambda_y=0$ for all $y\in\bbZ^d\setminus(\bbR x)$ except one $y$ such that $\lambda_y>0$. Up to applying a linear transformation, we may assume $x$ and $y$ to be the first two unit vectors of $\bbZ^2$. 

We next use a folklore proof to establish an isoperimetric inequality. Let $A\subset\bbZ^2$ and observe that
\begin{equation}
\label{eq:projections}
|A|\le |P_1(A)|\cdot|P_2(A)|,
\end{equation}
where $P_1(A)=\{a\in\bbZ:A\cap(\{a\}\times\bbZ)\neq\varnothing\}$ and $P_2(A)=\{b\in\bbZ:A\cap(\bbZ\times\{b\})\neq\varnothing\}$. Fix an integer $M\ge 1$ and set $\Lambda_M:=\sum_{z=1}^M\lambda_{(z,0)}$ and $C_M= 4 \sqrt{\Lambda_M \lambda_y}$. Looking at the right-most, left-most, top-most and bottom-most sites of $A$ on each horizontal or vertical line, we have 
\[\|\partial A\|_{\lambda^M}\ge 2|P_1(A)|\lambda_y+2|P_2(A)|\Lambda_M\ge C_M\sqrt{|A|},\]
the last inequality following from \eqref{eq:projections} and $(a + b)/2 \geq \sqrt{ab}$.

\noindent\textbf{Case 2.} Assume that for every $x\in\bbZ^d$, we have $\sum_{y\in\bbZ^d\cap(\bbR x)}\lambda_y<\infty$.
We start by noticing that $\| \partial A \|_\lambda$ is linear in $\lambda$.
Therefore, it is enough to find a decomposition $\lambda = \sum_{i \geq 1} \lambda^i$, with each $\lambda^i$ having finite support and such that, for every finite set $A \subset \bbZ^d$, we have $\| \partial A \|_{\lambda^i} \geq c_i \sqrt{|A|}$ with $C_M := \sum_{i \leq M} c_i \to \infty$ as $M \to \infty$.

The decomposition that we construct will be such that each $\lambda^i$ satisfies
\begin{gather}
  \text{$\lambda^i_z > 0$ if and only if $z \in \{x_i, y_i, -x_i, -y_i\}$,}\\
  \text{$x_i$ and $y_i$ are linearly independent and}\\
  \text{$\lambda^i_{x_i} = \lambda^i_{y_i}$.}
\end{gather}

We start by defining $\lambda^1$ and for this, pick some $x_1 \in \bbZ^d$ attaining the largest value of $\lambda_x$.
Recalling $\sum_{y \in \bbZ^d\setminus(\bbR x)} \lambda_y = \infty$, take $y_1$ attaining the largest $\lambda_y$ outside of $(\bbR x)$.
Finally, set $\lambda^{1}_{x_{1}} = \lambda^{1}_{y_{1}} = \lambda_{y_{1}}$ and zero otherwise.
At this point, we just repeat the above procedure with $\lambda$ replaced by the remainder $\lambda - \lambda^i$, inductively.
This defines all the $\lambda^i$'s and the fact that $\lambda = \sum_i \lambda^i$ follows from from the construction, together with the hypothesis of Case~2.
We are now left with proving that for every finite set $A \subset \bbZ^d$, we have $\| \partial A \|_{\lambda^i} \geq c_i \sqrt{|A|}$ with non-summable $c_i$'s.

Notice that, for each $i \geq 1$, the proof of the previous case gives that the boundary of $A$ corresponding to translates of $\{x_i, y_i\}$ alone satisfies $\|\partial A\|^{x,y}_{\lambda^i} \ge 4\sqrt{\smash{\lambda^i_{x_i} \lambda^i_{y_i}} |A|} = \lambda^i_{y_i} 4\sqrt{|A|}$, for any $A$ contained in the sub-lattice $\langle x_i, y_i \rangle_\bbZ$, so in fact the same holds for any $A \subset \bbZ^d$, since $\sqrt a + \sqrt b \ge \sqrt{a + b}$ for $a, b\ge 0$.
The fact that the values $\lambda^i_{y_i}$ are not summable follows from $\sum_i\lambda^i_{y_i}=\frac12 \sum_x \lambda_x = \infty$.
\end{proof}

\section{Proof of Theorem~\ref{th:transitive}}
\label{sec:transitive}

Theorem~\ref{th:transitive} will follow easily from Theorem~\ref{th:main} and the following lemma. 

\begin{lemma}[Isoperimetric inequality]\label{lem:transitive} 
There exists a universal constant $c_1>0$ such that the following holds. If $G=(V,E)$ is a transitive graph of degree $\Delta$ and superlinear growth, then 
$\psi_{\1_E}(n)\geq c_1\sqrt{\Delta n}$ for all $n\geq1$.
\end{lemma}

\begin{proof}
By \cite[Corollary 7.1]{TT24} for $d=1$, there exists a universal constant $\varepsilon_1>0$ such that every transitive graph $G$ of degree $\Delta$ and superlinear growth satisfies $|B_n|\geq \varepsilon_1 \Delta n^2$ for all $n\geq1$. By \cite[Proposition 5.1]{TT20}, this implies the isoperimetric inequality $|\partial S|\geq \frac{\sqrt{2\varepsilon_1\Delta}}{12} |S|^{1/2}$. Setting $c_1=\sqrt{\varepsilon_1/72}$ concludes the proof.
\end{proof}

\begin{proof}[Proof of Theorem~\ref{th:transitive}] 
First, we observe that for small $n$ we can easily obtain a better bound on $\psi_{\1_E}$. If $S\subset V$ is such that $|S|\leq \Delta/2$, then each vertex $x\in S$ has at least $\Delta/2$ neighbours outside $S$ and therefore $\|\partial S\|_{\1_E}=|\partial S| \geq \frac{\Delta}{2}|S|$. In particular, $\psi_{\1_E}(n)\geq \frac{\Delta}{2}n$ for all $n\leq \Delta/4$ (this is where we use the fact that we work with $\psi$ instead of $\phi$). 
Let $\lambda=\frac{C}{\Delta} \1_E$ for some large constant $C$ to be chosen later. Obviously $\psi_\lambda=\frac{C}{\Delta}\psi_{\1_E}$, so we have $\psi_{\lambda}(n)\geq \frac{C}{2}n$ for $n\leq \Delta/4$, and $\psi_{\lambda}(n)\geq c_1C\sqrt{n/\Delta}$ for $n\geq \Delta/4$, by Lemma~\ref{lem:transitive}. 
Therefore
\begin{align*}
  \sum_{k\geq0} \frac{1}{\psi_{\lambda}(2^k)} \leq \sum_{k=0}^{ \lfloor \log_2 (\Delta/4) \rfloor} \frac{2}{C2^k} + \frac{\sqrt{\Delta}}{c_1C}\sum_{k\geq\lfloor\log_2 (\Delta/4)\rfloor+1} \frac{1}{2^{k/2}} \leq \frac{4}{C} + \frac{2}{c_1C} \left( \frac{\sqrt 2}{\sqrt 2 -1} \right) < \frac{4(1+2c_1^{-1})}{C}.
\end{align*}
By setting $C:=2500 (1+2c_1^{-1})$, Theorem~\ref{th:main} gives $\bbP_{\lambda}[0\lr \infty]>0$. Recalling the parametrization $p_e=1-e^{-\lambda_e}$, this readily implies that $p_{\mathrm{c}}^{\mathrm{bond}}(G)\leq 1-e^{-C/\Delta}$.
\end{proof}

\begin{remark}[Optimality]
 The bound $\phi(n)\geq c\sqrt{\Delta n}$ from Lemma~\ref{lem:transitive} is optimal up to constant. Indeed, consider the transitive (in fact, also Cayley) graph $G$ with vertex set $V= \bbZ^2 \times (\bbZ/k\bbZ)$ and edge set $E=\{\{(x,j),(y,k)\}: \text{either } x=y \text{ or } j=k \text{ and } y=x\pm e_i, i\in\{1,2\}\}$, where $e_1=(1,0)$ and $e_2=(0,1)$. In this case, one has $\Delta= k+3$ and the sets $S_n=[-\sqrt{n/k},\sqrt{n/k}]^2\times (\bbZ/k\bbZ)$ satisfy $|S_n|\asymp n$ and $|\partial S_n|\asymp \sqrt{kn}$. On the other hand, for the specific case of Cayley graphs of $\bbZ^2$, a better bound of the form $\phi(n)\geq c'\Delta\sqrt{ n}$ can be proved via elementary arguments similar to those used in the proof of Lemma~\ref{prop:truncation}.
\end{remark}

\section*{Acknowledgements} 
This project started during the XXIX Brazilian School of Probability in Rio de Janeiro, where Bernardo de Lima brought the truncation problem to our attention and later also shared helpful bibliographic pointers. We would like to thank him as well as the organizers of the event. We also thank Lyuben Lichev and Matthew Tointon for helpful reference pointers. This work was carried out in part during a visit by FS to IMPA, whose hospitality he gratefully acknowledges. 
AT was supported by the grants ``Projeto Universal'' (406250/2016-2) from CNPq, ``Produtividade em Pesquisa'' (304437/2018-2) from CNPq and ``Cientista do Nosso Estado'' (204.377/2024) from FAPERJ.
For the purpose of Open Access, a CC-BY public copyright licence has been applied by the authors to the present document and will be applied to all subsequent versions up to the Author Accepted Manuscript arising from this submission.

\paragraph{Declaration on the use of AI.} No artificial intelligence tools were used in the development of the mathematical ideas, arguments, or proofs presented in this paper, nor were they used for drafting or writing the manuscript.

\appendix
\section{Proof of Lemma~\ref{lem:gpt}}
\label{sec:app}
This appendix entirely follows \cite{Cut26b} up to minor adjustments and is presented only to keep the proof self-contained. We introduce the decreasing function 
\[f(u)=\log(eu)/u, ~~~ u\ge 1.\]
\begin{claim}
\label{claim:computation}
    Let $l$ be a positive integer and $(p_i)_{i=0}^l$ and $(b_i)_{i=0}^l$ be sequences of positive real numbers. Assume that $b_i\ge \max(1,p_i)$ for all $i$ and $b_{i}\ge b_{i-1}+p_{i-1}/2$ for $i\in\{1,\dots,l\}$. Then
    \[\sum_{i=0}^lf(b_i) \le\frac{3\sqrt5}{\sqrt2}\left(\sum_{i=0}^l\frac{1}{p_i}\right)^{1/2}.\]
\end{claim}
\begin{proof}
For any $u\in J_i:=[b_i,b_i+p_i/2)$, we have $f(u)\ge f(3b_i/2)\ge\frac23f(b_i)$. Therefore,
\[p_if(b_i)^2\le \frac92\int_{J_i}f^2.\]
Since the intervals $J_i$ are disjoint, this yields 
\[\sum_{i=0}^lp_if(b_i)^2\le \frac92\int_1^\infty f^2=\frac{45}2.\]
The proof then follows by the Cauchy--Schwarz inequality $(\sum x_iy_i)^2\le(\sum x_i^2) (\sum y_i^2)$ with $x_i=\sqrt{p_i}f(b_i)$ and $y_i=1/\sqrt {p_i}$.
\end{proof}

\begin{proof}[Proof of Lemma~\ref{lem:gpt}]
Recall $a_t$, $b_t$ $r_t$ from \eqref{eq:def:rtbt} and note that we always have $b_t\ge\phi_\lambda(1) \ge 1$. Let $m_t=|v_{G_t}|_{G_t}$ be the size of the vertex selected at step $t$ in $\expl_\pi$. 

Let us fix $z\in V$ such that there exists $e\in\pi$ for which $z\in e$. Let $(t_i)_{i=0}^l$ be the times such that $v_{G_{t_i}}=z$.\footnote{We abused the notation by writing $v_{G_{t_i}}=z$ to represent that $z$ is mapped to $v_{G_{t_i}}$ via the natural quotient.} 
For all $i\in\{0,\dots,l-1\}$, let \[p_i=\min\left(\psi_\lambda^*(2^{\lfloor\log_2m_{t_{i}}\rfloor}),\psi_\lambda^*(2^{\lfloor\log_2 (m_{t_{i+1}}-m_{t_i})\rfloor})\right).\]
Recalling \eqref{eq:def:psi:star}, we have $b_{t_i}\ge p_i$. Moreover, by Definition~\ref{def:select}, $m_{t_{i}}\ge 2m_{t_{i-1}} \ge 2^i$ for all $i\in\{1,\dots,l\}$, so the integers $\lfloor\log_2 m_{t_i}\rfloor$ are all distinct. Consequently, $\lfloor \log_2 (m_{t_{i+1}}-m_{t_i})\rfloor$ can take the same integer value for at most two indices $i$, since $\lfloor\log_2 m_{t_i}\rfloor\leq \lfloor\log_2 (m_{t_{i+1}}-m_{t_i})\rfloor \leq \lfloor\log_2 m_{t_{i+1}}\rfloor$.  Hence, 
\begin{equation}
\label{eq:sum:inverse}\sum_{i=0}^{l-1}\frac1{p_i}\le 3\sum_{k\ge0}
\frac1{\psi^*_\lambda(2^k)}=\frac{3}{\Psi^*}
.
\end{equation}

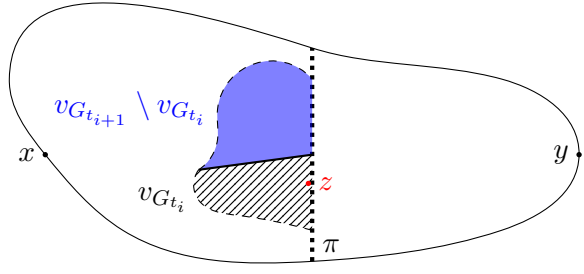
\begin{figure}
    \centering
    \begin{tikzpicture}[x=5cm,y=2cm]
    \draw (0,0) to [closed, curve through={(0.5,-0.205)..(1,0)..(1.205,0.5)..(1,1)..(0.5,1.205)..(-0.3,1)}] (-0.205,0.5);
    \fill (-0.205,0.5) node[left]{$x$} circle (1pt);
    \fill (1.205,0.5) node[left]{$y$} circle (1pt);
    \draw[ultra thick,dotted] (0.5,-0.205)--(0.5,1.205);
    \fill[pattern=north east lines] (0.5,0) to [curve through={(0.3,0.1)..(0.2,0.2)}] (0.2,0.4)--(0.5,0.5)--cycle;
    \fill[opacity=0.5,blue] (0.5,1) to [curve through={(0.25,0.8)..(0.25,0.6)}] (0.2,0.4)--(0.5,0.5)--cycle;
    \draw[thick](0.2,0.4)--(0.5,0.5);
    \draw[dashed](0.5,0) to [curve through={(0.3,0.1)..(0.2,0.2)}] (0.2,0.4);
    \draw[dashed](0.2,0.4) to [curve through={(0.25,0.6)..(0.25,0.8)}] (0.5,1);
    \draw (0.5,-0.1) node[right]{$\pi$};
    \fill[red] (0.49,0.31) node[right]{$z$} circle (1pt);
    \draw (0.2,0.2) node[left]{$v_{G_{t_i}}$};
    \draw[blue] (0.25,0.8) node[left]{$v_{G_{t_{i+1}}}\setminus v_{G_{t_i}}$};
    \end{tikzpicture}
    \caption{Illustration of the proof of Lemma~\ref{lem:gpt}. The cutset $\pi$ separating $x$ and $y$ is shown in the middle. The vertex set in $V$ corresponding to $v_{G_{t_i}}$ and containing $z$ is hatched. The shaded blue region is $v_{G_{t_{i+1}}}\setminus v_{G_{t_i}}$. Their common boundary is thickened and has weight at most $r_{t_i}$ (which is the total boundary of the hatched region, excluding $\pi$). The boundary of $v_{G_{t_{i+1}}}$ in $G_{t_{i+1}}$ is dashed and has weight $r_{t_{i+1}}$.}
    \label{fig:proof_addendum}
\end{figure}

Then the boundary of the vertex set corresponding to $v_{G_{t_{i+1}}}\setminus v_{G_{t_i}}$ in the multi-graph $G_{t_i}$ (which has total vertex weight $m_{t_{i+1}}-m_{t_i}$) fall in one of the categories (see Figure~\ref{fig:proof_addendum}):
\begin{itemize}
    \item edges in $G_{t_i}$ linking $v_{G_{t_i}}$ and $v_{G_{t_{i+1}}}\setminus v_{G_{t_i}}$ or 
    \item edges in $G_{t_{i+1}}\setminus \pi$ incident to $v_{G_{i+1}}$ or 
    \item edges in $\pi$ incident to $v_{G_{t_{i+1}}}\setminus v_{G_{t_i}}$ (but not to $v_{G_{t_i}}$).
\end{itemize}
The total weight of the above three types of edges is at most $r_{t_i}$, $r_{t_{i+1}}$ and $a_{t_{i+1}}-a_{t_i}$ respectively. Therefore,
\[b_{t_{i+1}}-b_{t_i} \overset{\eqref{eq:def:rtbt}}=(r_{t_i}+r_{t_{i+1}}+(a_{t_{i+1}}-a_{t_i}))-2r_{t_i}\ge
\psi^*_\lambda(2^{\lfloor \log_2(m_{t_{i+1}}-m_{t_i}) \rfloor})-2r_{t_i}\ge p_{i}-2r_{t_i}.\]
We consider four types of indices $i$ (recall that $b_{t_i} \geq p_i\geq \phi_\lambda(1)\geq 1$):
\begin{itemize}
  \item If $i<l$ and $r_{t_i}\ge p_{i}/4$, then $\log\frac{b_{t_i}}{\max(r_{t_i},1)} \le \log \frac{b_{t_i}}{r_{t_i}} =\log \big( 1 + \frac{a_{t_i}}{r_{t_i}} \big) \le \frac{a_{t_i}}{r_{t_i}}\le 4\frac{a_{t_i}}{p_{i}}$.
  \item If $i<l$ and $r_{t_i}<p_{i}/4$, then $\log\frac{b_{t_i}}{\max(r_{t_i},1)} \le \log b_{t_i}\le \log(eb_{t_i})\le \frac43a_{t_i}f(b_{t_i})$ and $b_{t_{i+1}}\geq b_{t_i} + p_i/2$.
  \item If $i=l$, and $r_{t_l}\ge b_{t_l}/2$, then $\log\frac{b_{t_i}}{\max(r_{t_i},1)} \le  \log \frac{b_{t_i}}{r_{t_i}}\le 2\frac{a_{t_i}}{b_{t_i}}\le 2a_{t_i}f(b_{t_i})\le 2a_{t_i}f(\phi_\lambda(1))$.
  \item If $i=l$ and $r_{t_l}<b_{t_l}/2$, then $\log\frac{b_{t_i}}{\max(r_{t_i},1)} \le \log b_{t_i}\le \frac 32 a_{t_i}f(b_{t_i})\le 2a_{t_i}f(\phi_\lambda(1))$.
\end{itemize}
Putting these bounds together and using Claim~\ref{claim:computation}, we obtain
\begin{equation}
  \label{eq:gpt:bound}
  \sum_{i=0}^l\frac{1}{a_{t_i}}\log\frac{b_{t_i}}{\max(r_{t_i},1)}\le 4\sum_{i=0}^{l-1}p_{i}^{-1}+2\sqrt{10}\left(\sum_{i=0}^{l-1}p_{i}^{-1}\right)^{1/2} + 2f(\phi_\lambda(1))
  \overset{\eqref{eq:sum:inverse}}\le 
  \frac{\delta}2.\end{equation}
Hence, recalling that $a_t=\sum_{e\in\pi,v_{G_t\in e}}\lambda_e$, we get
\[\sum_t\log\frac{b_t}{\max(r_t,1)}=\sum_t\frac{1}{a_t}\log\frac{b_t}{\max(r_t,1)}\sum_{e\in \pi}\lambda_e\sum_{z\in e}\1_{v_{G_t}=z}\le \delta\sum_{e\in\pi}\lambda_e=\delta\|\pi\|_{\lambda},\]
concluding the proof of Lemma~\ref{lem:gpt}.
\end{proof}

\end{document}